\documentclass[12pt]{amsart}
\usepackage{amsmath}
\usepackage{amssymb}
\usepackage{amsfonts}
\usepackage{amsthm}
\usepackage{verbatim}
\usepackage{amscd}
\usepackage{cite}
\usepackage{leftidx}
\usepackage{enumerate}
\usepackage{txfonts}
\usepackage{manfnt}
\usepackage{amscd}
\usepackage[mathscr]{eucal}
\usepackage{hyperref}
\usepackage{mathpazo}

\def\inv{^{-1}}

\DeclareMathOperator{\del}{\partial}

\DeclareMathOperator{\dlog}{dlog}

\def\refp #1.{(\ref{#1})}

\newcommand{\ul}[1]{\underline {#1}}

\def\sbr #1.{^{[#1]}}
\def\sfl #1.{^{\lfloor #1\rfloor}}

\def\inv{^{-1}}
\def\?{{\bf{??}}}

\def\dgla{differential graded Lie algebra\ }

\def\HH{\mathbb H}

\def\C{\mathbb C}
\def\P{\mathbb P}

\def\Q{\mathbb Q}

\def\O{\mathcal O}

\def\g{\mathfrak g}

\def\m{\mathfrak m}

\def\1/2{\frac{1}{2}}

\def\Ann{\textrm{Ann}}

\def\2{{[2]}}

\def\nl{\newline}

\def\<{\langle}
\def\>{\rangle}

\def\2{{[2]}}

\def\scl #1.{^{\lceil#1\rceil}}
\def\spr #1.{^{(#1)}}
\def\sbc #1.{^{\{#1\}}}

\def\subpr#1.{_{(#1)}}

\def\beq{\begin{equation*}}
\def\eeq{\end{equation*}}

\newcommand{\sing}{{\mathrm{sing}}}

\newcommand{\llog}[1]{\langle\log {#1} \rangle}
\newcommand{\llogp}[1]{\langle\log  {^+#1} \rangle}

\newcommand{\llogpp}[1]{\langle\log  {^{++} #1} \rangle}
\newcommand{\mlog}[1]{\langle-\log {#1} \rangle}
\def\g3{{\Gamma\spr 3.}}

\newcommand{\eqspl}[2]{
\begin{equation}\label{#1}
\begin{split}
#2\end{split}\end{equation}}

\newcommand{\beginalphaenum}{
\begin{enumerate}\renewcommand{\labelenumi}{ }
\item \begin{enumerate}
}

\def\eex{\end{rm}\end{example}}

\def\pisharp{\Pi^\sharp}
\def\piflat{\Pi^\flat}

\newtheorem{thm}{Theorem} 

\newtheorem*{thm*}{Theorem}
\newtheorem*{prop*}{Proposition}
\newtheorem{cor}[thm]{Corollary}
\newtheorem*{cor*}{Corollary}

\newtheorem{lem}[thm]{Lemma}

\newtheorem*{claim*}{Claim}

\newtheorem{defn}[thm]{Definition}
\theoremstyle{remark}

\newtheorem{rem}[thm]{Remark}
\newtheorem*{rem*}{Remark}
\newtheorem{crit-rem}[thm]{Critical remark}

\newtheorem{example}[thm]{Example}
\newtheorem*{example*}{Example}
\newtheorem*{lem*}{Lemma}

\newtheorem*{defn*}{Definition}

\begin{document} 
\title{A Bogomolov unobstructedness theorem\\  for
 log-symplectic manifolds in general
position}
\author 
{Ziv Ran}


\date {\today}


\address {\nl UC Math Dept. \nl
Big Springs Road Surge Facility
\nl
Riverside CA 92521 US\nl 
ziv.ran @  ucr.edu\nl
\url{http://math.ucr.edu/~ziv/}
}


\begin{abstract}
We consider compact K\"ahlerian manifolds $X$ 
of even dimension 4 or more, endowed with
a log-symplectic holomorphic Poisson structure 
$\Pi$ which is sufficiently general,
in a precise linear sense, with respect to its (normal-crossing)
degeneracy divisor $D(\Pi)$. We prove that $(X, \Pi)$ has unobsrtuced
deformations, that the tangent space to
 its deformation space can be identified in terms
of the mixed Hodge structure on $H^2$ of the open symplectic
manifold  $X\setminus D(\Pi)$, and in fact coincides with this $H^2$
provided the Hodge number $h^{2,0}_X=0$, and finally
that the degeneracy locus $D(\Pi)$
deforms locally trivially under deformations of $(X, \Pi)$.
\end{abstract}
\maketitle
We consider here holomorphic Poisson manifolds $(X, \Pi)$, i.e. complex manifolds $X$ 
(generally, compact and K\"ahlerian) endowed
with a holomorphic Poisson structure $\Pi$. We say that$(X, \Pi)$ is  \emph{log-symplectic}
if $X$ has even dimension $2n$ and the degeneracy or Pfaffian divisor $D(\Pi)$ of $\Pi$,
i.e. the divisor of the section $\Pi^n\in H^0(X, \wedge^{2n}T_X)=H^0(X, -K_X)$,
is a reduced divisor with (local) normal crossings (NB: there are other notions of log-symplectic
in the literature). $(X, \Pi)$ is \emph{simply log-symplectic}
if moreover $D(\Pi)$ has simple normal crossings, i.e. is a 
transverse union of smooth
components. We note that Lima and Pereira \cite{lima-pereira} have shown that if
$(X, \Pi)$ is simply log symplectic and $X$ is a Fano manifold of dimension
4 or more with cyclic Picard group, then $X$ is $\P^{2n}$ with a standard 
(toric)
Poisson stucture \[\Pi=\sum a_{ij}x_ix_j\del/\del{x_i}\del/\del{x_j},\]
$x_i$ homogeneous coordinates (and consequently deformations of 
$(X, \Pi)$ are of the same kind, at least set-theoretically).
\par
A natural question about Poisson manifolds is that of understanding
 deformations of the pair $(X, \Pi)$; in particular
whether $(X, \Pi)$ has unobstructed deformations, and whether
such deformations change (e.g. smooth out) the singularities of the
degeneracy locus $D(\Pi)$ or whether, on the contrary, $D(\Pi)$
deforms locally trivially.
When $(X, \Pi)$ is symplectic, i.e. $D(\Pi)=0$, unobtructedness
is of course well known (Bogomolov, Tian, Todorov).
In \cite{qsymplectic} we proved unobstructedness and local triviality
 when $\Pi$ satisfies
a special condition called $P-normality$ which states that $D$ is smooth at every
point where the corank of $\Pi$ is exactly 2. This condition
 is equivalent to saying that
$\Pi$ has the largest possible corank, viz. $2m$, at a point of given multiplicity
$m$ on $D$ (or equivalently, the smallest multiplicity for given corank).
P-normality is also equivalent to $(X, \Pi)$ being decomopsable
 locally as the 
product of Poisson surfaces with smooth or empty degeneracy divisor.\par
Our interest here is in Poisson structures with a property that goes in the
opposite direction to, and indeed excludes, P-normality. This property,
that we call $r$-general position and need just for $r=2$, 
states that $\Pi$ is 'general' in a neighborhood of $D$, in the sense
that the columns of  its matrix with respect to
a coordinate system adapted to $D$ satisfy a linear general position
property. The 2-general position property implies that 
$(X, \Pi)$ cannot split off locally a factor equal to a Poisson surface with
singular (normal-crossing) degeneracy divisor. Our main result (Theorem
\ref{mainthm}) states that if $\Pi$ is in 2-general position (which implies that 
$2n=\dim(X)\geq 4$), then $(X, \Pi)$ has unobstructed deformations
and moreover, these deformations induce locally trivial deformations
of the degeneracy divisor $D(\Pi)$. In fact, the
 tangent space to the deformation space
can be identified locally in terms of the (mixed) Hodge theory
of the open symplectic variety $U=X\setminus D(\Pi)$,
namely is equals $F^1H^2(U, \C)$. 
Whenever $h^{2,0}_X=0$, the latter coincides with $H^2(U, \C)$.
For $X=\P^{2n}$ and $\Pi$ a toric Poisson structure,
this is due on the set-theoretic level to 
Lima- Pereira \cite{lima-pereira}. Here we give
an extension of the Lima-Pereira result to the toric case: 
we will determine the deformations of 2-general log-symplectic
toric Poisson structures
on smooth projective toric varieties,
showing in particular that they remain toric (cf. Corollary \ref{toric}).
\par
 The local triviality conclusion is surprising
because it is decidedly false in dimension 2. Though in \cite{qsymplectic}
we proved an analogous local triviality result for deformations
of P-normal 
Poisson structures, 
the latter certainly does hold in dimension 2 (where the degenracy locus
must be smooth).
While the 2-general position- or something like it- seems necessary
for local triviality, it's unclear whether it is necessary for mere
unobstructedness. In the case of Poisson surfaces, i.e. surfaces endowed with an
effective anticanonical divisor, deformations are unobstructed whenever the divisor is
reduced and has normal crossings, but can be obstructed
when the divisor is non-reduced (see \cite{qsymplectic}).
\par
The strategy for the proof is simple. Poisson deformations are 'controlled'
by the Poisson-Schouten \dgla (dgla)
$(T^\bullet_X, [\ .\ ,\Pi])$. As shown in \cite{qsymplectic}, 
using only that $D$ has normal crossings,
the sub-dgla
$T^\bullet_X\mlog{D}$ has unobstructed deformations thanks
to its isomorphism with the dgla of log forms 
$(\Omega^\bullet_X\llog{D}, d)$. 
We extend this isomorphism to
an isomorphism between $T^\bullet_X$ and a certain dgla of 
meromorphic forms denoted $\Omega^\bullet_X\llogp{D}$,
the 'log-plus forms', 
which contains the log forms. We identify the quotient complex
$\Omega^\bullet_X\llogp{D}/\Omega^\bullet_X\llog{D}$ and prove-
under the 2-general position hypothesis-
that it is exact is  degrees $\leq 2$.\par
Other unobstructedness results of 'Bogomolov-Tian-Todorov type'
were obtained  by Kontsevich et al. 
\cite{kontsevich-gen-tt}, \cite{ka-kontsevich-pa}. Earlier, Hitchin \cite{hitchin-poisson}
and Fiorenza-Manetti \cite{fiorenza-manetti} had proven unobsrtuctedness for certain
deformation directions, namely, those corresponding to the K\"ahler class itself. 
Also, Pym \cite{pym} has introduced the notion of elliptic Poisson structures
and more generally, Pym and Schedler \cite{pym-schedler} then introduced the notion of \emph{holonomic}
Poisson manifolds, where the degeneracy divisor is reduced but not necessarily
normal crossings, and have studies their deformations focusing on
questions of local finite dimensionality. \par
The referee points out the work
of M\v arcut- Osorno Torres \cite{marcut-ot} in the real case, which considers
deformations of real  $C^\infty$ Poisson structures with \emph{smooth} degeneracy locus
(of real codimension 1) . Though the setting is different, their
methods, especially in their \S 3,
 are somewhat related to those of our \S 2 below. The analogue of their
 quasi-isomorphism result (Lemma 3) also holds in the holomorphic case
 (assuming smooth degeneracy divisor).\par 
In this paper we work in the holomorphic category and use the complex topology unless otherwise mentioned.
\subsection*{Acknowledgment} I am grateful to the IHES for its hospitality 
(and peerless ambiance) during preparation of this paper.
 Thanks are due to
the referee for helpful and interesting comments.
\section{Basics}
See \cite{dufour-zung} or \cite{ciccoli}  or \cite{pym-constructions} for basic facts on 
Poisson manifolds and \cite{ginzburg-kaledin} (especially the appendix),
\cite{goto-rozanski-witten},
or \cite{namikawa}, and references therein
 for information on deformations of
Poisson complex structures.\par
\subsection{log vector fields, duality}
Consider a log-symplectic $2n$-dimensional Poisson manifold $(X, \Pi)$ with degeneracy or Pfaffian divisor
$D=[\Pi^n]$. We denote by $\pisharp$  the interior multiplication map
\[\pisharp:\Omega^1_X\to T_X.\]
As $\pisharp$ is generically an isomorphism, it extends to an isomorphism on
meromorphic differentials. 
The following result
is apparently well known and was stated in 
\cite{qsymplectic}, Lemma 6 under the needlessly strong hypothesis
that $\Pi$ is P-normal.
\begin{lem}\label{pisharp} 
If $D$ has local normal crossings, then $\pisharp$ yields an isomorphism
\eqspl{pisharp}{\pisharp:\Omega^1_X\llog{D}\to T_X\mlog{D}.}
\end{lem}\label{loc-trivial}
\begin{proof}
The assertion is obvious at points where $\Pi$ is nondegenerate.
As smooth points of $D$, it follows easily from Weinstein's normal form.
Thus, $\pisharp$ is a morphism of locally free sheaves on $X$ which
is an isomorphism off a codimension-2 subset, viz. $\sing(D)$.
Therefore $\pisharp$ is an isomorphism.
\end{proof}
\subsection{log deformations}
Note that $\Pi$ can be viewed as a section of 
\[T^2_X\mlog{D}:=\wedge^2T_X\mlog{D},\] i.e.
as "Poisson structure for the Lie algebra sheaf $T_X\mlog{D}$"; one
can think of the Lemma as saying that $\Pi$ is \emph{nondegenerate
as such}. 
Note that a $T_X\mlog{D}$-deformation, say
over a local artin algebra $R$, consists of a flat 
deformation of $X, \tilde X/R$ plus
a compatible locally trivial deformation of $D$, $\tilde D/R$.
On the other hand a $T^\bullet_X\mlog{D}$-deformation over
 $R$, consists of a flat 
deformation of $X, \tilde X/R$,
a compatible locally trivial deformation of $D$, $\tilde D/R$,
plus a compatible deformation $\tilde\Pi$ of $\Pi$, that is,  a section of
$T^2_{\tilde X}\mlog{\tilde D}$ extending $\Pi$. In that case clearly
the degeneracy locus $D(\tilde \Pi)=\tilde D$. 
Thus the natural forgetful map 
\[T^\bullet_X\mlog{D}\to T_X\mlog{D}\]
corresponds to mapping
\[\tilde\Pi\to D(\tilde\Pi).\]
We record this for future
reference:
\begin{lem}Let $(X, \Pi)$ be a log-symplectic manifold with degeneracy 
locus $D=D(\Pi)$ and let $R$ be a local Artin algebra. Then an $R$-vaued $T^\bullet_X\mlog{D}$-deformation induces\par (i) a flat deformation
$\tilde X/R$;\par
(ii) an $R$-linear extension $\tilde\Pi$ of $\Pi$ to $\tilde X$,
such that the degeneracy locus $D(\tilde\Pi)$ is a locally trivial
deformation of $D$.
\end{lem}
\begin{rem}
Any deformation of $(X, \Pi)$ induces a deformation of $D(\Pi)$ but the 
latter deformation need not be locally trivial in general (this is already clear when $X$ is 2-dimensional,
so $\Pi$ corresponds to an effective anticanonical divisor which can be singular and deform to
a nonsingular one).
\end{rem}
It also follows from Lemma \ref{pisharp} , and was also noted in
 \cite{qsymplectic},  that $\pisharp$ extends to an isomorphism
of differential graded algebras
\eqspl{pisharp-graded}{\pisharp:(\Omega^\bullet_X\llog{D}, d)\to (T^\bullet_X\mlog{D}, [.,\Pi]).}
We will denote the graded  isomorphism inverse to $\pisharp$ by $\piflat$.
Working locally in a neighborhood of a point, let $D_1,..., D_m$
be the components of $D$, with respective local equations
$x_1,..., x_m$ which are part of a local coordinate system $x_1,...,x_{2n}$.
Thus, denoting $\del_i=\del/\del x_i$, we have a local basis for $T_X\mlog{D}$ of the form
\[v_i:=\begin{cases}x_i\del_i, \  i=1,...,m\\
\del_{i},\  i=m+1, ...,2n,
\end{cases}\]
with the dual basis
\[\eta_i:=\begin{cases}dx_i/x_i,\ i=1,...,m,\\
dx_{i}, \  i=m+1, ..., 2n\end{cases}\]
being a local basis for $\Omega^1_X\llog{D}$. Let $A=(a_{ij})$
be the matrix of $\pisharp$ with respect to these
bases and $B=(b_{ij})=A\inv$. Thus
$A$ and $B$ are skew-symmetric.
Then we have
\[\Pi=\sum\limits_{1\leq i<j\leq 2n}a_{ij}v_iv_j,\]
\[\Phi:=\Pi\inv=\sum\limits_{1\leq i<j\leq 2n}b_{ij}\eta_i\eta_j.\]
Note that the map $\piflat$ is just interior multiplication by $\Phi$.\par
\subsection{General position}\label{general-position-sec}
A pair of $k\times k$ matrices $M, N$ with entries in a 
ring $R$ are said to be in
$t$-\emph{relative general position} if every collection of $t$ columns
of  $M$, together with the corresponding columns of $N$,
 generate a free and cofree rank-$2t$ submodule of $R^k$;
equivalently, for every such 'matched' collection of $2t$ columns
from $M$ and $N$,  there exists a collection of $2t$
rows such that the corresponding $2t\times 2t$ minor is a unit
(we will henceforth abbreviate the latter conclusion by simply saying
that the columns in question are linearly independent).
Clearly this condition is never satisfied if $2t>k$ while if $2t\leq k$
it is satisfied if $N$ is nonsingular and $M$ is general relative to $N$.
$M$ is said to be in (standard) \emph{$t$-general position}
if $M, I_k$ are in relative $t$-general position. 
The log-symplectic
structure $\Pi$ as above is said to be in $t$-general position
if, locally at each point of the degeneracy locus, there is a 
local coordinate system as above such that
 the matrix $A$ above is in (standard) $t$-general position. 
In general, 
if $A$ is 'generic', 
$\Pi$ will be in $n$-general position.
On the other hand, if $(X, \Pi)$ decomposes as a product structure with a factor of
dimension $2k$ then $\Pi$ is not in $(2k)$-general position.

\section{The Log-plus complex}
\subsection{Definition of log-plus complex}
Notations as above, we also denote by $\Omega^1_X(*D)$ the sheaf of meromorphic forms
with arbitrary-order pole along $D$:
\[\Omega^1_X(*D)=\bigcup\limits_{k=1}^\infty \Omega^1_X(kD).\]
Then $\piflat$ extends to an injective map $T_X\to\Omega^1_X(*D)$, whose
image we denote by $\Omega^1_{X, \Pi}\llogp{D}$. 
As subsheaf of $\Omega^1_X(*D)$, this sheaf depends on $\Pi$, but we will often abuse
notation and denote it simply by $\Omega^1_X\llogp{D}$.
Thus, $\piflat$ yields an isomorphism
\[\piflat:T_X\to\Omega^1_X\llogp{D}.\]
Also set
\[\Omega^k_X\llogp{D}=\wedge^k\Omega^1_X\llogp{D}.\]
Note that the isomorphism of complexes
\[\piflat:(\wedge^\bullet T_X\mlog{D}, [\ .\ ,\Pi])\to(\Omega_X^\bullet\llog{D}, d)\]
now extends to an isomprhism of complexes
\eqspl{}{
\piflat:(\wedge^\bullet T_X, [\ .\ ,\Pi])\to(\Omega_X^\bullet\llogp{D}, d).
}
The log-plus complex has appeared before, in a different context, in
the work of Polishchuk \cite{polishchuk-ag-poisson}.\par
\subsection{Comparison with log complex}
Denote by $\Omega^1_X\llogpp{D}$ the subsheaf of $\Omega^1_X(*D)$, 
generated by $\Omega^1_X\llog{D}$ and $\eta_j/x_i, i\leq m, j\neq i$.
Then using, e.g. \eqref{pisharp}, note that $\Omega^1_X\llogp{D}$ is a 
subsheaf of
  $\Omega^1_X\llogpp{D}$.
Also set
\[\Omega^k_X\llogpp{D}=\wedge^k\Omega^1_X\llogpp{D}.\]
Now let $(\phi_\bullet)=\piflat(\del_\bullet)$ be the local basis of
$\Omega^1_X\llogp{D}$ corresponding to the basis $\del_1,...,\del_{2n}$
of $T_X$. Thus
\[\phi_i=\sum\limits_{j=1}^m b_{ij}dx_j/x_ix_j+\sum\limits_{j=m+1}^{2n}b_{ij}dx_j/x_i, i=1,...,m\]
while $\phi_{m+1},...,\phi_{2n}\in\Omega^1_X\llog{D}$. Note that for $i=1,...,m$, $x_i\phi_i$
corresponds to the log vector field $v_i=x_i\del_i$, which is automatically
locally Hamiltonian off the degeneracy locus, hence $x_i\phi_i$ is a closed form. Therefore
\[d\phi_i=\phi_i\wedge dx_i/x_i, i=1,...,m.\]
More generally, for a multi-index $I=(i_1<...<i_k)$, setting
\[\phi_I=\phi_{i_1}\wedge...\wedge\phi_{i_k},\]
we have
\eqspl{dphi}{
d\phi_I=\sum\limits_{j=1}^k((-1)^jdx_{i_j}/x_{i_j})\phi_I.
}
This implies immediately that $(\Omega^\bullet_X\llogp{D},d)$ is a complex,
i.e. that $d\Omega^k_X\llogp{D}\subset \Omega_X^{k+1}\llogp{D}$.
Moreover, $\Omega^\bullet_X\llogp{D}$ admits an ascending filtration
$\mathcal F_\bullet$ by subcomplexes (and $\Omega^\bullet_X\llog{D}$- modules under
exterior product) defined by:\par
 - $\mathcal F_0=\Omega^\bullet_X\llog{D}$;\par
- for each $i\geq 0, \mathcal F_{i+1}$ is generated by $\mathcal F_i$ and $\Omega^1_X\llogp{D}\mathcal F_i$.\par
Thus, $\mathcal F_i$ is generated over $\Omega_X^\bullet\llog{D}$
by $\Omega_X^j\llogp{D}, j=1,...,i$. In particular, the quotient $\mathcal G_i:=\mathcal F_i/\mathcal F_{i-1}$
is generated over $\Omega^\bullet_X\llog{D}$ by the images of the  $\phi_I, |I|=i$, which we denote by
$\bar\phi_I$. Note that
$\Ann(\bar\phi_I)=(x_r:r\in I)$. Let $D_I$ be the corresponding stratum. Then $\mathcal G_i$
decomposes at a direct sum of complexes of locally free $\O_{D_I}$-modules: 
\eqspl{}{
\mathcal G_i=\bigoplus\limits_{|I|=i}\bar\phi_I\Omega^\bullet_X\llog{D}[-i]\otimes\O_{D_I}.
} The directness of the sum follows from the correponding
assertion for $\wedge^\bullet T_X\mlog{D}$, which is obvious.\par
We denote the summand above that corresponds to $I$ by $Q^\bullet_I$,
a complex in degrees $\geq i$.\par
Note that
\[\Omega^1_X\llog{D}\otimes\O_{D_I}=\Omega^1_{D_I}\llog {\bar D_I}\oplus\bigoplus\limits_{r\in I}(dx_r/x_r)\O_{D_I}\]
where $\bar D_I$ is the normal-crossing divisor on$D_I$ induced by $D$ (i.e. by the components of $D$
not containing $D_I$).
This decomposition  induces an analogous one on
the exterior power  $\Omega^k_X\llog{D}\otimes\O_{D_I}$
We now introduce the hypothesis that $\Pi$ is in
$t$-general position. This implies that
for any index-set $I\subset \ul m$ with  $|I|\leq t$, 
any set of $2t$ log-forms among $\{dx_i/ x_i,x_i \phi_i: \in I\}$ is
linearly independent. 
Equivalently, the $x_i\phi_i, i\in I$
pull back to a collection of nonvanishing linearly independent log forms on $D_I$
provided $|I|\leq t$. 
Note also that
\[(\sum\limits_{r=1}^m\O_X\phi_r)\cap\Omega_X^1\llog{D}=\sum\limits_{r=1}^m\O_Xx_i\phi_r\]
(just apply the isomorphism $\pisharp$ to both sides).
Consequently, we have,
\emph{provided $\Pi$ is in $t$-general position}, that
\eqspl{}{
Q_I^{|I|+k}\simeq\bigoplus\limits_{j=0}^k\Omega^{k-j}_{D_I}\llog{\bar D_I}\left(\bigoplus\limits_{\stackrel{|J|=j}{J\subset I}}d\log(x)_J\right),
|I|+k\leq t,}
where
\[d\log(x)_J:=\wedge_{i\in J}dx_i/x_i.\]
$ t$-general position is needed to insure that any collection of $t$ distinct $dx_j/x_j$-s and $x_i\phi_i$-s is independent,
i.e. spans a free and cofree submodule of $\Omega^1_X\llog{D}$. 
and consequently for every $d\log(x)_J$ appearing above,
$d\log(x)_J\wedge\phi_I $ is nonvanishing.
Note that in terms of the above identification, the differential on $\Q^\bullet_I$ is given by
\eqspl{}{
d(\psi_Jd\log(x)_J)=(d\psi_J)d\log(x)_J+\sum\limits_{i\in I\setminus J}\psi_J(-1)^id\log(x)_Jdx_i/x_i.
} Of course $d\log(x)_J$ and $d\log(x)_Jdx_i/x_i$ are just place holders,
so dropping these, we get more concretely the formula for the 'twisted'
differential  is
\eqspl{}{
\tilde d(\psi_J)=(d\psi_J, -\psi_J)_{i\in I\setminus J}.
} Thus, for elements $\psi_J$ in the component of type $(k, J)$
(so that $\psi_J$ is a differential of degree $k-|J|$, the differential
has components of type $(k+1, J)$ (namely $d\psi_J$) and type
$(k+1, J\cup\{i\})$ (namely $\pm\psi_J$), for all
$i\in I\setminus J$, and they are given as above.\par
Note that such a complex is automatically exact except
if it consists of a single group: a cocycle
of the form $(\psi, \gamma)$ such that $\psi=d\gamma$ equals
$\tilde d(\pm \gamma. 0)$. Consequently, we conclude
\begin{lem}
If $\Pi$ is in $r$-general position, then for each $I$, $\Q_I^\bullet$
is exact in degrees $\leq r$.
\end{lem}
As an immdiate consequence, we conclude
\begin{cor}\label{qis}
If $\Pi$ is in $r$-general position, then the natural
map
\[\mathbb H^i(\Omega_X^\bullet\llog{D})\to 
\mathbb H^i(\Omega_X^\bullet\llogp{D})\]
is an isomorphism for $i\leq r$.
\end{cor}
\begin{rem}
In the real $C^\infty$ case considered in \cite{marcut-ot}, with smooth degeneracy locus, it is shown
in \S 3  that the
analogous map inclusion is a quasi-isomorphism. 
In the holomorphic case (with smooth degeneracy locus), the analogous fact is true as well: the inclusion map from
the log complex to the log-plus complex is a quasi-isomorphism, no general
position hypotheses needed. This follows easily from computations as above, given Weinstein's normal form
\cite{ciccoli}.
\end{rem}
\begin{rem}
A 2-dimensional log-symplectic Poisson structure $\Pi$ with \emph{singular} degeneracy 
divisor $D(\Pi)$ is never in 2-general position and
in fact the complex $\mathcal G_2$ above is not locally acyclic in degree 1 (i.e.
on 2-forms): for the Poisson structure
$\frac{dx_1dx_2}{x_1x_2}$, the 2-form $\frac{dx_1dx_2}{x_1^2x_2^2}$ is closed but not exact
in the log-plus complex.\par
Naturally the same applies whenever the log-symplectic manifold $(X,\Pi)$ splits as a (Poisson) product
with  a 2-dimensional factor $(X_1, \Pi_1)$
such that $D(\Pi_1)$ is singular.
\end{rem}
\section{Main theorem}
\subsection{Main result}
Our main result is the following
\begin{thm}\label{mainthm}
Let $(X, \Pi)$ be a log-symplectic compact K\"ahlerian 
Poisson manifold of dimension
$2n\geq 4$ with degeneracy
divisor $D$. Assume $(X, \Pi)$ is in 2-general position. Then
\par (i) $(X, \Pi)$ has
unobstructed deformations;\par
(ii) the tangent space to the deformation space of
$(X, \Pi)$ coincides with the Hodge-level subspace
$F^1\HH^2(X\setminus D, \C)$;\par
(iii) any deformation of $(X, \Pi)$ induces
a locally trivial deformation of the degeneracy locus $D(\Pi)$.
\end{thm}
\begin{proof}
As is well known, the deformation theory of $(X, \Pi)$ is the deformation 
theory of the \dgla sheaf $(T_X^\bullet, [\ .\ ,\Pi])$. The latter is isomorphic 
as dgla to $(\Omega_X^\bullet\llogp{D},d)$, so it suffices to prove that 
obstructions vanish for the latter. By Corollary \ref{qis}, the inclusion map
\[\Omega_X^\bullet\llog{D}\to\Omega_X^\bullet\llogp{D}\]
induces an isomorphism through degree 2. However, it was already 
observed in \cite{qsymplectic}, based on Hodge theory for the log 
complex and the isomorphism $\pisharp:\Omega^\bullet_X\llog{D}
\to T^\bullet_X\mlog{D}$, that the dgla $\Omega_X^\bullet\llog{D}$ has vanishing
obstructions. Since these obstructions only involve terms
of degree $\leq 2$, it follows that
$\Omega_X^\bullet\llogp{D}$ has vanishing obstructions as well.
This proves (i). Now (ii) follows from basic mixed Hodge theory,
specifically the fact that that the log complex 
of a divisor with normal crossings computes the cohomology
of the complement
\cite{deligne-hodge-ii} (see \cite{grif-schmid}, Sec 5, esp. Prop. 5.14), 
while (iii) follows from the fact observed
above that $T^\bullet_X\mlog{D}$-deformations induce locally trivial
deformations of $D$.
\end{proof}
\begin{cor}
Notations as above, assume additionally that $h^{2,0}_X=0$. Then the 
tangent space to the deformation space of $(X, \Pi)$ coincides with
$H^2(X\setminus D, \C)$.
\begin{proof}
Our assumption implies that $H^2(\Omega^0_X\llog{D})=0$, hence
by Deligne's mixed Hodge theory  as referenced above $F^1H^2(X\setminus D, \C)=H^2(X\setminus D, \C)$.
\end{proof}
\end{cor}
\subsection{Toric Poisson structures}
By way of example, we will apply Theorem \ref{mainthm} to toric
Poisson structures on toric varieties. We begin with some general remarks
on toric varieties. Let $X$ be a smooth projective toric variety 
of dimension $d$ with open orbit $U\simeq(\C^*)^d$ and boundary
$D=X\setminus U$,  a divisor with normal crossings. Then the invariant
differentials $dx_i/x_i$ pulled back from the $\C^*$ factors extend to
log differentials on $X$, and in particular,
\eqspl{geq}{h^0(\Omega_X^i\llog{D})\geq \binom{d}{i}.}
On the other hand, note that
\[U\simeq (\C^*)^{d}\sim_{\mathrm
{homotopy}}(S^1)^{d},\] hence $H^i(U, \C)$ is $\binom{d}{i}$-dimensional.
But  by basic mixed Hodge theory(see \cite{deligne-hodge-ii} or \cite{grif-schmid}),
the latter vector space has a filtration with quotients $H^{i-j}(\Omega_X^j
\llog{D})$. Comparing with \eqref{geq} we conclude
\eqspl{log-betti}{
h^i(\Omega_X^j\llog{D})=\begin{cases}0, i>0;\\
\binom{d}{j}, i=0.
\end{cases}
}
Thus,
\eqspl{zero}{
H^i(U, \C)\simeq H^0(\Omega^i_X\llog{D}).
}
Now let $v_1,..., v_d$ be vector fields on $X$
corresponding to the $(\C^*)^d$-
action. They are dual to the $dx_i/x_i$. 
Then for any skew-symmetric $d\times d$ matrix $A=(a_{ij})$
we get a $(\C^*)^d$-invariant  Poisson structure on $X$:
\eqspl{}{
\Pi_A=\sum a_{ij}v_i\wedge v_j.
} Now assume $d=2n$ even, and that $A$ is nonsingular and in 2-general
position. Then $(X, \Pi_A)$ is log-symplectic in 2-general position.
Note that 
\[H^0(T^2_X\mlog{D})\simeq H^0(\Omega^2_X\llog{D})\]
is generated by the $v_i\wedge v_j$. Therefore by applying Theorem
\ref{mainthm} we conclude
\begin{cor}\label{toric}
Notations as above, and assuming $A$ is nonsingular in 2-general position, 
$(X, \Pi)_A$ has unobstructed deformations
parametrized by a germ of $H^0(X, T^2_X\mlog{D})$,
all of the form $(X, \Pi_{A'})$ obtained by deformaing $A$, 
and the tangent space to the
deformation space coincides with $H^2(X\setminus D, \C)$ and is
$\binom{2n}{2}$-dimensional. In particular, all deformations of $(X, \Pi)$
are toric.
\end{cor}
When $n=1$, no 2-dimensional Poisson structure is in 2-general position
as has been noted above in \S \ref{general-position-sec}. For $n>1$, taking
$(a_{ij})$ general, or e.g. specifically  $a_{ij}=i-j$
yields 2-general position, while $A$ corresponding to a direct sum of hyperbolic planes does not.\par
In the case $X=\P^{2n}$ a similar result (for $A$ generic) was obtained
by Lima and Pereira \cite{lima-pereira}.


\vfill\eject

	\part*{Erratum/Corrigendum added October 2020\\
		'A BOGOMOLOV UNOBSTRUCTEDNESS THEOREM FOR \\
		LOG-SYMPLECTIC MANIFOLDS IN GENERAL POSITION'\\
		(J. Inst. Math. Jussieu 2018)}

	
	
	
	\begin{abstract}
		The general position hypothesis needs strengenthing
	\end{abstract}
	I am grateful to Brent Pym (pers. comm.) for showing me by 
	an example due jointly to him with Mykola Matviichuk and Travis Schedler 
	(see arxiv:math/201008692 and below) that the general position hypothesis
	in the paper is not strong enough. As their example shows, it is possible,
	even with
	the 2-general position hypothesis,
	for the inclusion of the log complex to 
	the log-plus complex
	not to induce a surjection on local and even global cohomology, and
	for the Poisson structure to admit deformations where the Pfaffian
	divisor deforms non locally-trivially.
	\par
	We are thus led to introduce a stronger general position condition  that we  call
	'very general position' on a Poisson structure. Throughout this note
	we keep the notations and assumptions of the paper.
	\begin{defn}
		The Poisson pair $(X, \Pi)$ is said to  be in $t$-very general position if
		it is in $t$-general position and
		for any index set $I\subset [m]$ with $|I|\leq t$, the standard basis vectors
		$e_1,...,e_m$ are linearly independent over $\Q$ modulo 
		$\sum\limits_{i\in I}\O_Xk_i$, where $k_i$ is the $i$th column of 
		the submatrix $(b_{ij}:i, j\in[m])$ of $B$.
	\end{defn}
	Using the notations of \S 2.2 of the paper, the significance of 
	the $t$-very general property
	is this.	
	For any $I\subset [m], |I|\leq t$,
	the local log 1-forms $\dlog(x_i), i=1,...,m$
	are linearly independent over $\Q$ modulo 
	$\sum\limits_{i\in I}x_i\phi_i\O_X+\Omega^1_X$. In other words,
	for any set of integers $r_1,...,r_m$, not all zero, we have

	\eqspl{}{\sum r_j\dlog(x_j)
		\not\in  \sum\limits_{i\in I}x_i\phi_i\O_X+\Omega^1_X.
	}
The following is a partial substitute for Lemma 4 of the paper:
\begin{lem*}
	If $(X,\Pi)$ is in $t$-very general position then $Q^\bullet_I$
	is exact in degrees $<t$.
	\end{lem*}
\begin{proof}	Continuing with the notations of the paper, let us further set
	\[\dlog(x)^K=\bigwedge\limits_{k\in K}\dlog(x_k),\ \ 
	x^K=\prod\limits_{k\in K} x_k.\]
	The complex denoted $Q_I$ is generated by subgroups
	\[Q_{I}^{J, K}=
	x^I\phi_I\dlog(x)^J\dlog(x)^K
	\O_{D_I}, \ \ {J\subset I, K\cap I=\emptyset}\]
	with the differential
	\[\delta_I=\wedge\dlog(x^I)+d\]
	where the first summand acts on the $\dlog(x^I)$ factor
	where $d$ acts solely on the $\O_{D_I}$ factor.
	Note that the two summands commute. Let $\hat\O_{D_I}$
	be the formal completion of $\O_{D_I}$ at the maximal ideal $\m$
	of the origin.
	As usual, flatness of $\hat\O_{D_I}$
	over $\O_{D_I}$ yields that it will suffice to
	prove that the formal completion $Q_I\otimes\hat{\O}_{D_I}$ is exact
		in degrees $<t$, i.e. that the sequence
		\[Q_I\otimes(\O_{D_I}/\m^{k+1})\to Q_I\otimes(\O_{D_I}/\m^{k})\to
		Q_I\otimes(\O_{D_I}/\m^{k-1})\]
	is exact in the middle in degrees $<t$.\par
	
	Now assume temporarily that
	the matrix $B=(b_{ij})$ is constant, i.e.
	$\Phi$ has constant
	coefficients in the $\dlog(x_i)$, 
	and hence so do the forms $x_i\phi_i=\sum b_{ij}\dlog(x_j)$.
	Then the complex admits a an action by the torus $T=\mathbb {G}_m^m$,
	which acts trivially on any $\dlog(x_i))$, hence acts only on the
	$\hat{\O}_{D_I}$ factor. Now $\O_{D_I}/\m^k$ 
	decomposes as finite direct product of character sheaves
	$\O_{D_I}^M$ indexed by nonnegative
	$M$-tuples $M$ with $M\cap I=\emptyset$ and $|M|<k$.
	This yields a similar decomposition of $Q_I^{J, K}$
	and $Q_I^\bullet$ into
	subsheaves $Q_I^{J, K, M}$ resp. $Q_I^{\bullet M}$
	as well. Moreover, on $Q_I^{\bullet M}$, $\delta_I$ is simply given by
	wedge product with 
	\[\eta:=\dlog(x^{I}x^{M}).\]
Note that, thanks to our
$t$-very general hypothesis,  $\eta$ 	
is a 'primitive' or cotorsion-free element of $Q^1_I$, i.e. maps to a nonzero
	element of the fibre of $Q^1_I$ at the origin. As is well known and
	easy to prove (e.g. by making $\eta$ part of a basis), 
	wedge product with 
	such an element $\eta$
	yields an exact complex $\wedge^\bullet Q^{1, M}_I=Q_I^{\bullet M}$.
	Therefore $Q_I$ is exact in degrees $< t$, in the case $B$ is constant. 
	
	%
	%
	In the general case we use semi-continuity: let
	$\alpha_s $ be multiplication by $s\in \C$ defined 
	locally near the origin, and set $\Phi_s=\alpha_s^*\Phi$.
	For $s=0$, $\Phi_s$ has constant coefficients, hence
	the corresponding complex is exact in degree $< t$.
	Then by semi-continuity the corresponding complex
	remains exact in degrees $< t$ for all small enough $s$. But since the
	complexes are equivalent for all $s\neq 0$, it follows
	that the original complex is exact in degrees $<t$. 
	
	\end{proof}


	
	As in the paper, we deduce:
	\begin{thm*}[Theorem 8 corrected]
		The conclusions of Theorem 8 hold under the additional hypothesis that
		$(X, \Pi)$ is in 2-very general position.
		
	\end{thm*}
	\begin{example}[M. Matviichuk, B. Pym, T. Schedler, arxiv 2010.08692]
		Consider the matrix
		\eqspl{}{
			B=(b_{ij})=\left (
			\begin{matrix}
				0&1&2&4\\
				-1&0&3&5\\
				-2&-3&0&6\\
				-4&-5&1&0
			\end{matrix}	\right)	
		}
		and the corresponding log-symplectic form on $\C^4$, $\Phi=\sum \limits_{i<j}
		b_{ij}\frac{dz_i}{z_i}\wedge \frac{dz_j}{z_j}$ and corresponding Poisson
		structure $\Pi=\Phi\inv$, both of which extend to $\P^4$ with Pfaffian divisor
		$D=(z_0z_1z_2z_3z_4)$, $z_0=$ hyperplane at infinity. 
		Then $\Pi$ admits the 1st order Poisson deformation
		with bivector $z_3z_4\del_{z_1}\del_{z_2}$, which in fact extends to
		a Poisson deformation of $(\P^4, \Pi)$
		over the affine line $\C$, and the Pfaffian divisor deforms 
		as $(z_3z_4z_0(z_1z_2-tz_3z_4))$, hence non locally-trivially.
		Correspondingly, the log-plus form $z_3z_4\phi_1\phi_2$ is closed (
		and not exact). That  $d(z_3z_4\phi_1\phi_2)=0$ corresponds to the 
		integral column relation
		\[k_1-k_2+(e_1+e_2)-(e_3+e_4)=0\]	
		where the $k_i$ and $e_j$ are the columns of 
		the $B$ matrix and the identity, respectively, showing that $\Pi$, though
		2-general is not 2-very
		general.  
	\end{example}
\vfill\eject
\bibliographystyle{amsplain}
\bibliography{../../../mybib}
\end{document}